\documentclass[a4paper, 12pt]{article}
\usepackage{amssymb}
\usepackage{amstext}
\usepackage{amsmath}
\usepackage{amscd}
\usepackage{amsthm}
\usepackage{latexsym}
\usepackage{amsfonts}
\usepackage[all]{xy}
\usepackage{enumerate}
\usepackage{textcomp}
\usepackage{color}
\usepackage{mathrsfs}

\numberwithin{equation}{section}
\usepackage{enumerate}

\theoremstyle{plain}
\newtheorem{Theorem}{Theorem}[section]

\newtheorem{lemma}[Theorem]{Lemma}

\theoremstyle{definition}

\begin{document}
\pagestyle{plain}
\setlength{\baselineskip}{20pt}

\title{Non-central limit theorems for quadratic functionals of Hermite-driven long memory moving average processes}
\author{T. T. Diu Tran\footnote{Facult\'{e} des Sciences, de la Technologie et de la Communication, UR en math\'{e}matiques. Universit\'{e} du Luxembourg, 6 rue Richard Coudenhove-Kalergi, L-1359 Luxembourg. E-mail: diu.tran@uni.lu}}
\maketitle

\begin{abstract} 
Let $(Z_t^{(q, H)})_{t \geq 0}$ denote a Hermite process of order $q \geq 1$ and self-similarity parameter $H  \in (\frac{1}{2}, 1)$. Consider the Hermite-driven moving average process
$$X_t^{(q, H)} = \int_0^t x(t-u) dZ^{(q, H)}(u), \qquad t \geq 0.$$ 
In the special case of $x(u) = e^{-\theta u}, \theta > 0$, $X$ is the non-stationary Hermite Ornstein-Uhlenbeck process of order $q$. Under suitable integrability conditions on the kernel $x$, we prove that as $T \to \infty$, the normalized quadratic functional
$$G_T^{(q, H)}(t)=\frac{1}{T^{2H_0 - 1}}\int_0^{Tt}\Big(\big(X_s^{(q, H)}\big)^2 - E\Big[\big(X_s^{(q, H)}\big)^2\Big]\Big) ds , \qquad t \geq 0,$$
where $H_0 = 1 + (H-1)/q$, converges in the sense of finite-dimensional distribution to the Rosenblatt process of parameter $H' =  1 + (2H-2)/q$, up to a multiplicative constant, irrespective of self-similarity parameter whenever $q \geq 2$. In the Gaussian case $(q=1)$, our result complements the study started by Nourdin \textit{et al} in \cite{Ivan2}, where either central or non-central limit theorems may arise depending on the value of self-similarity parameter. A crucial key in our analysis is an extension of the connection between the classical multiple Wiener-It\^{o} integral and the one with respect to a random spectral measure (initiated by Taqqu (1979)), which may be independent of interest. 
\end{abstract}



\section{Motivation and main results}

Let $(Z_t^{(q, H)})_{t \geq 0}$ be a Hermite process of order $q \geq 1$ and self-similarity parameter $H  \in (\frac{1}{2}, 1)$. It is a $H$-self-similar process with stationary increments, exhibits long-range dependence and can be expressed as a multiple Wiener-It\^{o} integral of order $q$ with respect to a two-sided standard Brownian motion $(B(t))_{t \in \mathbb{R}}$ as follows:
\begin{equation}\label{eq:H1}
Z^{(q, H)}(t) = c(H, q) \int_{\mathbb{R}^q} \bigg( \int_0^t \prod_{j=1}^q(s- \xi_j)_+^{H_0 - \frac{3}{2}}ds\bigg) dB(\xi_1)\ldots dB(\xi_q),
\end{equation}
where
\begin{equation}\label{eq:H2}
c(H, q) = \sqrt{\frac{H(2H - 1)}{q! \beta^q(H_0 - \frac{1}{2}, 2-2H_0)}} \qquad \text{and} \qquad H_0 = 1+\frac{H-1}{q}.
\end{equation}
Particular examples include the fractional Brownian motion $(q=1)$ and the Rosenblatt process $(q=2)$. For $q \geq 2$, it is no longer Gaussian. All Hermite processes share the same basic properties with fractional Brownian motion such as self-similarity, stationary increments, long-range dependence and even covariance structure. The Hermite process has been pretty much studied in the last decade, due to its potential to be good model for various phenomena.

A theory of stochastic integration with respect to $Z^{(q, H)}$, as well as stochastic differential equation driven by this process, have been considered recently. We refer to \cite{Ivan, Nualart} for a recent account of the fractional Brownian motion and its large amount of applications. We refer to \cite{Taqqu2, Tudor2,Tudor} for different aspects of the Rosenblatt process. Furthermore, in the direction of stochastic calculus, the construction of Wiener integrals with respect to $Z^{(q, H)}$ is studied in \cite{Tudor1}. According to this latter reference, stochastic integrals of the form
\begin{equation}\label{eq:12}
\int_{\mathbb{R}}f(u)dZ^{(q, H)}(u)
\end{equation}
are well-defined for elements of $\mathcal{H} = \{f: \mathbb{R} \to \mathbb{R}: \int_{\mathbb{R}}\int_{\mathbb{R}} f(u)f(v)|u-v|^{2H-2}dudv < \infty\}$, endowed with the norm
\begin{equation}\label{eq:14}
||f||_{\mathcal{H}}^2 = H(2H-1)\int_{\mathbb{R}}\int_{\mathbb{R}} f(u)f(v)|u-v|^{2H-2}dudv.
\end{equation}
Moreover, when $f \in \mathcal{H}$, the stochastic integral (\ref{eq:12}) can be written as
\begin{equation}\label{eq:13}
\int_{\mathbb{R}}f(u)dZ^{(q, H)}(u) = c(H, q) \int_{\mathbb{R}^q}\bigg(\int_{\mathbb{R}}f(u) \prod_{j=1}^q(u- \xi_j)_+^{H_0 - \frac{3}{2}}du\bigg)dB(\xi_1)\ldots dB(\xi_q)
\end{equation}
where $c(H, q)$ and $H_0$ are as in (\ref{eq:H2}). Since the elements of $\mathcal{H}$ may be not functions but distributions (see \cite{Nualart}), it is more practical to work with the following subspace of $\mathcal{H}$, 
which is a set of functions:
$$ |\mathcal{H}| = \bigg\{f: \mathbb{R} \to \mathbb{R}: \int_{\mathbb{R}}\int_{\mathbb{R}} |f(u)||f(v)||u-v|^{2H-2}dudv < \infty \bigg\}.$$

Consider the stochastic integral equation
\begin{equation}\label{eq:SDE}
X(t) = \xi  - \lambda \int_0^tX(s)ds + \sigma Z^{(q, H)}(t), \qquad t \geq 0,
\end{equation}
where $ \lambda, \sigma > 0$ and where the initial condition $\xi$ can be any random variable. By \cite[Prop. 1]{Tudor1}, the unique continuous solution of (\ref{eq:SDE}) is given by
$$X(t) = e^{-\lambda t} \bigg( \xi + \sigma \int_0^t e^{\lambda u} dZ^{(q, H)}(u) \bigg), \qquad t \geq 0.$$
In particular, if  for $\xi$ we choose $\xi = \sigma \int_{-\infty}^0 e^{\lambda u} dZ^{q, H}(u)$, then
\begin{equation}\label{eq:17}
X(t) = \sigma \int_{-\infty}^t e^{-\lambda(t-u)}dZ^{(q, H)}(u), \qquad t\geq 0.
\end{equation}
According to \cite{Tudor1}, the process $X$ defined by (\ref{eq:17}) is referred to as the Hermite Ornstein-Uhlenbeck process of order $q$. On the other hand, if the initial condition $\xi$ is set to be zero, then the unique continuous solution of (\ref{eq:SDE}) is this time given by
\begin{equation}\label{eq:18}
X(t) = \sigma \int_0^t e^{-\lambda(t-u)}dZ^{(q, H)}(u), \qquad t\geq 0.
\end{equation}
In this paper, we call the stochastic process (\ref{eq:18}) the \textit{non-stationary} Hermite Ornstein-Uhlenbeck process of order $q$. It is a particular example of a wider class of moving average processes driven by Hermite process, of the form
\begin{equation}\label{eq:Xt}
X_t^{(q, H)} := \int_0^t x(t-u) dZ^{(q, H)}(u), \qquad t \geq 0.
\end{equation}

In many situations of interests (see, e.g., \cite{Viens1, Viens2}), we may have to analyze the asymptotic behavior of the \textit{quadratic functionals of $X_t^{(q, H)}$} for statistical purposes. More precisely, let us consider
\begin{equation}\label{eq:Gt}
G_T^{(q, H)}(t):=\frac{1}{T^{2H_0 - 1}}\int_0^{Tt}\Big(\big(X_s^{(q, H)}\big)^2 - E\Big[\big(X_s^{(q, H)}\big)^2\Big]\Big) ds.
\end{equation}
In this paper, we will show that $G_T^{(q, H)}$  converges in the sense of finite-dimensional distribution to the Rosenblatt process (up to a multiplicative constant), irrespective of the value of $q \geq 2$ and $H \in (\frac{1}{2}, 1)$. The case $q=1$ is apart, see Theorem \ref{Theorem2} below.

\begin{Theorem}\label{Theorem1}
Let $H \in (\frac{1}{2}, 1)$ and let $Z^{(q, H)}$ be a Hermite process of order $q \geq 2$ and self-similarity parameter $H$. Consider the Hermite-driven moving average process $X^{(q, H)}$ defined by (\ref{eq:Xt}), and assume that the kernel $x$ is a real-valued integrable function on $[0, \infty)$ satisfying, in addition,
\begin{equation}\label{eq:1}
\int_{\mathbb{R}_+^2} |x(u)||x(v)||u-v|^{2H-2}dudv < \infty.
\end{equation}
Then, as $T \to \infty$, the family of stochastic processes $G_T^{(q, H)}$ converges in the sense of finite-dimensional distribution to $b(H, q)R^{H'}$, where $R^{H'}$ is the Rosenblatt process of parameter $H' =  1 + (2H-2)/q$ (which is the \emph{second-order} Hermite process of parameter $H'$), and the multiplicative constant $b(H, q)$ is given by
\begin{equation}\label{eq:19}
b(H, q) = \frac{H(2H-1)}{\sqrt{(H_0 -\frac{1}{2})(4H_0 - 3)}}\int_{\mathbb{R}_+^2}x(u)x(v)|u-v|^{(q-1)(2H_0 -2)}dudv.
\end{equation}
(The fact that (\ref{eq:19}) is well-defined is part of the conclusion of the theorem.)
\end{Theorem}
Theorem \ref{Theorem1} only deals with $q \geq 2$, because $q=1$ is different. In this case, $Z^{(1, H)}$ is nothing but the fractional Brownian motion of index $H$ and $X^{(1, H)}$ is the fractional Volterra process, as considered by Nourdin, Nualart and Zintout in \cite{Ivan2}. In this latter reference, a Central Limit Theorem for $G_T^{(1, H)}$ has been established for $H \in (\frac{1}{2}, \frac{3}{4})$. Here, we rather study the situation where $H \in (\frac{3}{4}, 1)$ and, in contrast to \cite{Ivan2}, we show a Non-Central Limit Theorem. More precisely, we have the following theorem. 

\begin{Theorem}\label{Theorem2}
Let $H \in (\frac{3}{4}, 1)$. Consider the fractional Volterra process $X^{(1, H)}$ given by (\ref{eq:Xt}) with $q=1$. If the function $x$ defining $X^{(1, H)}$ is an integrable function on $[0, \infty)$ and satisfies (\ref{eq:1}), then the family of stochastic processes $G_T^{(1, H)}$ converges in the sense of finite-dimensional distribution, as $T \to \infty$, to the Rosenblatt process $R^{H''}$ of parameter $H'' = 2H-1$ multiplied by $b(1, H)$ as above.
\end{Theorem}

It is worth pointing out that, irrespective of the self-similarity parameter $H \in (\frac{1}{2}, 1)$, the normalized quadratic functionals of any non-Gaussian Hermite-driven long memory moving average processes $(q \geq 2)$ exhibits a convergence to a random variable belonging to the second Wiener chaos. It is in strong contrast with what happens in the Gaussian case $(q=1)$, where either central or non-central limit theorems may arise depending on the value of the self-similarity parameter.

We note that our Theorem \ref{Theorem2} is pretty close to Taqqu's seminal result \cite{Taqqu1975}, but cannot be obtained as a consequence of it. In contrast, the statement of Theorem \ref{Theorem1} is completely new, and provides new hints on the importance and relevance of the Rosenblatt process in statistics.

Our proofs of Theorems \ref{Theorem1} and \ref{Theorem2} are based on the use of chaotic expansions into multiple Wiener-It\^{o} integrals and the key transformation lemma from the classical multiple Wiener-It\^{o} integrals into the one with respect to a random spectral measure (following a strategy initiated by Taqqu in \cite{Taqqu}). Let us sketch them. Since the random variable $X_t^{(q, H)}$ is an element of the $q$-th Wiener chaos, we can firstly rely on the product formula for multiple integrals to obtain that the quadratic functional $G_T^{(q, H)}(t)$ can be decomposed into a sum of multiple integrals of even orders from $2$ to $2q$. Secondly, we prove that the projection onto the second Wiener chaos converges in $L^2(\Omega)$ to the Rosenblatt process: we do this by using its spectral representation of multiple Wiener-It\^{o} integrals and by checking the $L^2(\mathbb{R}^2)$ convergence of its kernel. Finally, we prove that all the remaining terms in the chaos expansion are asymptotically negligible. 

Our findings and the strategy we have followed to obtain them owe a lot and were influenced by several seminal papers on Non-Central Limit Theorems for functionals of Gaussian (or related) processes, including Dobrushin and Major \cite{Dobrushin Major}, Taqqu \cite{Taqqu} and most recently, Clausel \textit{et al} \cite{CATudor1, CATudor2} and Neufcourt and Viens \cite{Viens}.
 
Our paper is organised as follows. Section $2$ contains preliminary key lemmas. The proofs of our two main results, namely Theorems \ref{Theorem1} and \ref{Theorem2}, are then provided in Section $3$ and Section $4$. 

\section{Preliminaries}
Here, we mainly follow Taqqu \cite{Taqqu}. We describe a useful connection between multiple Wiener-It\^{o} integrals with respect to random spectral measure and the classical stochastic It\^{o} integrals. Stochastic representations of the Rosenblatt process are then provided at the end of the section.

\subsection{Multiple Wiener-It\^{o} integrals with respect to Brownian motion}

Let $f \in L^2(\mathbb{R}^q)$ and let us denote by $I_q^B(f)$ the $q$th multiple Wiener-It\^{o} integral of $f$ with respect to the standard two-sided Brownian motion $(B_t)_{t \in \mathbb{R}}$, in symbols
$$I_q^B(f) = \int_{\mathbb{R}^q} f(\xi_1, \ldots, \xi_q) dB(\xi_1)\ldots dB(\xi_q).$$
When $f$ is symmetric, we can see $I_q^B(f)$ as the following iterated adapted It\^{o} stochastic integral:
$$I_q^B(f) = q! \int_{-\infty}^{\infty}dB(\xi_1) \int_{-\infty}^{\xi_1}dB(\xi_2) \ldots \int_{-\infty}^{\xi_{q-1}}dB(\xi_q) f(\xi_1, \ldots, \xi_q).$$
Moreover, when $f$ is not necessarily symmetric one has $I_q^B(f) = I_q^B(\widetilde{f})$, where $\widetilde{f}$ is the symmetrization of $f$ defined by
\begin{equation}\label{eq:dola}
\widetilde{f}(\xi_1, \ldots, \xi_q) = \frac{1}{q!}\sum_{\sigma \in \mathfrak{S}_q} f(\xi_{\sigma(1)}, \ldots, \xi_{\sigma(q)}).
\end{equation}
The set of random variables of the form $I_q^B(f), f \in L^2(\mathbb{R}^q)$, is called the $q$th Wiener chaos of $B$. We refer to Nualart's book \cite{Nualart} (chapter 1 therein) or Nourdin and Peccati's books \cite{Ivan, PeccatiIvan} for a detailed exposition of the construction and properties of multiple Wiener-It\^{o} integrals. Here, let us only recall the product formula between two multiple integrals: if $f \in L^2(\mathbb{R}^p)$ and $g \in L^2(\mathbb{R}^q)$ are two symmetric functions then
\begin{equation}\label{eq:P1}
I_p^B(f)I_q^B(g) = \sum_{r=0}^{p \wedge  q} r!\binom{p}{r}\binom{q}{r}I_{p+q-2r}^B(f \widetilde{\otimes}_r g),
\end{equation}
where the contraction $f \otimes_r g$, which belongs to $L^2(\mathbb{R}^{p+q-2r})$ for every $r = 0, 1, \ldots, p \wedge q$, is given by
\begin{align}\label{eq:P2}
f \otimes_r g& (y_1, \ldots, y_{p-r}, z_1, \ldots, z_{q-r}) \nonumber\\ 
& = \int_{\mathbb{R}^r} f(y_1, \ldots, y_{p-r}, \xi_1, \ldots, \xi_r) g(z_1, \ldots, z_{q-r}, \xi_1, \ldots, \xi_r) d\xi_1 \ldots d\xi_r
\end{align}
and where a tilde denotes the symmetrization, see (\ref{eq:dola}). Observe that
\begin{equation}\label{eq:3}
\| f \widetilde{\otimes}_r g\|_{L^2(\mathbb{R}^{p+q-2r})} \leq \| f \otimes_r g\|_{L^2(\mathbb{R}^{p+q-2r})} \leq \|f\|_{L^2(\mathbb{R}^p)}\|g\|_{L^2(\mathbb{R}^q)}, \quad r= 0, \ldots, p\wedge q
\end{equation}
by Cauchy-Schwarz inequality, and that $f \otimes_p  g = \left\langle f, g \right\rangle_{L^2(\mathbb{R}^p)}$ when $p=q$.
Furthermore, we have the orthogonality property
$$E[I_p^B(f)I_q^B(g)] =
\begin{cases}
 & p! \big\langle \widetilde{f}, \widetilde{g} \big\rangle_{L^2(\mathbb{R}^p)} \qquad\text{if } p=q\\
& 0 \qquad\qquad\qquad\quad \text{ if } p \ne q.
\end{cases}$$
% Finally, we state another useful property for elements belonging to a Wiener chaos, the so-called \textit{hypercontractivity} property, which we are going to use in the proof of the tightness.
% \begin{Theorem}
% Let $f \in L^2(\mathbb{R}^q)$ with $q \geq 1$. Then, for all $k \in [2, \infty)$,
% \begin{equation}\label{eq:9}
% E[|I_q^B(f)|^k]^{1/k} \leq (k-1)^{q/2}E[|I_q^B(f)|^2]^{1/2}.
% \end{equation}
% \end{Theorem}
\subsection{Multiple Wiener-It\^{o} integrals with respect to a random spectral measure}

Let $W$ be a Gaussian complex-valued random spectral measure that satisfies $E[W(A)] = 0, E[W(A)\overline{W(B)}] = \mu(A \cap B), W(A) = \overline{W(-A)} $ and $W(\bigcup_{j=1}^n A_j) = \sum_{j=1}^n W(A_j)$ for all disjoint Borel sets that have finite Lebesgue measure (denoted here by $\mu$). The Gaussian random variables $\text{Re}W(A)$ and $\text{Im}W(A)$ are then independent with expectation zero and variance $\mu(A)/2$. We now recall briefly the construction of multiple Wiener-It\^{o} integrals with respect to $W$, as defined in Major \cite{Major} or Section $4$ of Dobrushin \cite{Dobrushin}. To define such stochastic integrals let us introduce the real Hilbert space $\mathscr{H}_m$ of complex-valued symmetric functions $f(\lambda_1, \ldots, \lambda_m), \lambda_j \in \mathbb{R}, j=1, 2, \ldots, m$, which are even, i.e. $f(\lambda_1, \ldots, \lambda_m) = \overline{f(-\lambda_1, \ldots, -\lambda_m)}$, and square integrable, that is, 
$$\|f \|^2 = \int_{\mathbb{R}^m}|f(\lambda_1, \ldots, \lambda_m)|^2 d\lambda_1\ldots d\lambda_m < \infty.$$
The scalar product is similarly defined: namely, if $f, g \in \mathscr{H}_m$, then
$$\left\langle f, g \right\rangle_{\mathscr{H}_m} = \int f(\lambda_1, \ldots, \lambda_m)\overline{g(\lambda_1, \ldots, \lambda_m)}d\lambda_1 \ldots d\lambda_m.$$
The integrals $I_m^W$ are then defined through an isometric mapping from $\mathscr{H}_m$ to $L^2(\Omega)$:
$$ f \longmapsto I_m^W(f)  = \int_{\mathbb{R}}f(\lambda_1, \ldots, \lambda_m) W(d\lambda_1)\ldots W(d\lambda_m),$$
Following \emph{e.g.} the lecture notes of Major \cite{Major}, if $f \in \mathscr{H}_m$ and $g \in \mathscr{H}_n$, then $E[I_m^W(f)] = 0$ and
\begin{equation}\label{eq:21}
E[I_m^W(f)I_n^W(g)] =
\begin{cases}
 & m!\left\langle f, g \right\rangle_{\mathscr{H}_m} \text{ if } m=n\\
& 0 \qquad\qquad\quad\text{if } m \ne n.
\end{cases}
\end{equation}

\subsection{Preliminary lemmas}
We recall a connection between the classical Wiener-It\^{o} integral $I^B$ and the one with respect to a random spectral measure $I^W$ that will play an important role in our analysis.

\begin{lemma}\cite[Lemma 6.1]{Taqqu}\label{Lemma6.1} Let $A(\xi_1, \ldots, \xi_m)$ be a real-valued symmetric function in $L^2(\mathbb{R}^m)$ and let
\begin{equation}\label{eq:Fourier}
\mathcal{F}A(\lambda_1, \ldots, \lambda_m) = \frac{1}{(2\pi)^{m/2}}\int_{\mathbb{R}^m}e^{i\sum_{j=1}^m \xi_j\lambda_j}A(\xi_1,\ldots, \xi_m)d\xi_1\ldots d\xi_m
\end{equation}
be its Fourier transform. Then
$$\int_{\mathbb{R}^m}A(\xi_1,\ldots, \xi_m)dB(\xi_1) \ldots dB(\xi_m) \overset{(d)}{=} \int_{\mathbb{R}^m}\mathcal{F}A(\lambda_1,\ldots, \lambda_m)W(d\lambda_1)\ldots W(d\lambda_m).$$
\end{lemma}

Applying Lemma \ref{Lemma6.1}, we deduce the following lemma which is an extended result of Lemma 6.2 in \cite{Taqqu}.

\begin{lemma}\label{Lemma2}
Let 
$$ A(\xi_1, \ldots, \xi_{m+n}) = \int_{\mathbb{R}^2} \phi(z_1, z_2)\prod_{j=1}^{m}(z_1 - \xi_j)_+^{H_0 -\frac{3}{2}}\prod_{k=m+1}^{m+n}(z_2 - \xi_k)_+^{H_0 -\frac{3}{2}}dz_1dz_2$$
where $\frac{1}{2} < H_0 < 1$ and where $\phi$ is an integrable function on $\mathbb{R}^2$  whose Fourier transform is given by (\ref{eq:Fourier}).
Let 
$$\widetilde{A}(\xi_1, \ldots, \xi_{m+n}) = \frac{1}{(m+n)!}\sum_{\sigma \in \mathfrak{S}_{m+n}}A(\xi_{\sigma(1)}, \ldots, \xi_{\sigma(m+n)})$$
be the symmetrization of $A$. Assume that 
$$\int_{\mathbb{R}^{m+n}}|\widetilde{A}(\xi_1, \ldots, \xi_{m+n})|^2d\xi_1\ldots d\xi_{m+n} < \infty.$$
 Then,
\begin{align*}
&\int_{\mathbb{R}^{m+n}}\widetilde{A}(\xi_1,\ldots, \xi_{m+n})dB(\xi_1)\ldots dB(\xi_{m+n})\\ 
&\overset{(d)}{=} \bigg(\frac{\Gamma(H_0 - \frac{1}{2})}{\sqrt{2\pi}}\bigg)^{m+n} \int_{\mathbb{R}^{m+n}}W(d\lambda_1) \ldots W(d\lambda_{m+n}) \prod_{j=1}^{m+n} |\lambda_j|^{\frac{1}{2} - H_0}\\
&\qquad\qquad\quad\qquad\times\frac{1}{(m+n)!}\sum_{\sigma \in \mathfrak{S}_{m+n}} 2\pi\mathcal{F}\phi(\lambda_{\sigma(1)}+\ldots+ \lambda_{\sigma(m)}, \lambda_{\sigma(m+1)} + \ldots + \lambda_{\sigma(m+n)}).
\end{align*}
\end{lemma}

\begin{proof}
Thanks to Lemma \ref{Lemma6.1}, we first estimate the Fourier transform of $A(\xi_1, \ldots, \xi_{m+n})$. Because the function $u_+^{H_0 - \frac{3}{2}}$ belongs neither to $L^1(\mathbb{R})$ nor to $L^2(\mathbb{R})$, by similar arguments as in the proof of \cite[Lemma 6.2]{Taqqu} let us introduce
$$A_T(\xi_1, \ldots, \xi_{m+n})= \begin{cases}
&A(\xi_1, \ldots \xi_{m+n}) \text{ if } |\xi_j| < T \text{ }\forall  j =1, \ldots, m+n.\\
&0 \qquad\qquad\qquad\text{ otherwise.}
\end{cases}
$$
Set
$$B_\lambda(a, b) = \frac{1}{\sqrt{2\pi}}\int_a^b e^{-iu\lambda}u^{H_0 - \frac{3}{2}}du$$
for $0 \leq a \leq b < \infty$, and $B_\lambda(a, \infty)  = \lim_{b \to \infty}B_\lambda(a, b)$. By \cite[page 80]{Taqqu}, we get 
$$\sup_{0 \leq a \leq b}|B_\lambda(a, b)| \leq \frac{1}{\sqrt{2\pi}}\bigg(\frac{1}{H_0 - \frac{1}{2}} + \frac{2}{|\lambda|}\bigg).$$
Now, 
\begin{align*}
&\mathcal{F}A_T(\lambda_1,\ldots, \lambda_{m+n}) = \frac{1}{(\sqrt{2\pi})^{m+n}}\int_{\mathbb{R}^{m+n}}d\xi_1\ldots d\xi_{m+n} e^{i \sum_{j=1}^{m+n}\lambda_j\xi_j} \int_{\mathbb{R}^2}dz_1dz_2 \phi(z_1, z_2)\\
&\qquad\quad\qquad\qquad\qquad\qquad\times\prod_{j=1}^m(z_1 - \xi_j)_+^{H_0 - \frac{3}{2}}\prod_{j=m+1}^{m+n}(z_2 - \xi_j)_+^{H_0 - \frac{3}{2}}\mathbf{1}_{\{|\xi_j|<T, \forall j =1,\ldots, m+n\}}.
\end{align*}
The change of variables $\xi_j = z_1 - u_j$ for $j=1, \ldots, m$ and $\xi_j = z_2 - u_j$ for $j=m+1, \ldots, m+n$ yields
\begin{align*}
&\mathcal{F}A_T(\lambda_1,\ldots, \lambda_{m+n})\\
&= \frac{1}{(\sqrt{2\pi})^{m+n}}\int_{\mathbb{R}^{m+n}}du_1 \ldots du_{m+n} e^{-i \sum_{j=1}^{m+n}\lambda_j u_j} \int_{\mathbb{R}^2}dz_1dz_2 \phi(z_1, z_2)e^{i \sum_{j=1}^m\lambda_jz_1 }e^{i \sum_{j=m+1}^{m+n}\lambda_jz_2}\\
&\qquad\qquad\qquad\times \prod_{j=1}^{m} u_j^{H_0 - \frac{3}{2}}\mathbf{1}_{\{u_j > 0\}}\mathbf{1}_{\{ z_1 - T < u_j < z_1 + T\}}\prod_{j=m+1}^{m+n} u_j^{H_0 - \frac{3}{2}}\mathbf{1}_{\{u_j > 0\}}\mathbf{1}_{\{ z_2 - T < u_j < z_2 + T\}}.
\end{align*}
Suppose that $\lambda_1,\ldots, \lambda_{m+n}$ are different from zero. Since $\phi$ is integrable on $\mathbb{R}^2$ then
\begin{align*}
|\mathcal{F}A_T&(\lambda_1,\ldots, \lambda_{m+n})|\\
&\leq \int_{\mathbb{R}^2}dz_1dz_2 |\phi(z_1, z_2)| \prod_{j=1}^m B_{\lambda_j}(\max(0, z_1-T), \max(0, z_1+T))\\ 
&\qquad\qquad\qquad\quad\quad\times \prod_{j=m+1}^{m+n} B_{\lambda_j}(\max(0, z_2-T), \max(0, z_2+T))\\
& \leq \int_{\mathbb{R}^2}dz_1dz_2 |\phi(z_1, z_2)| \prod_{j=1}^{m+n}\frac{1}{\sqrt{2\pi}}\bigg(\frac{1}{H_0 -\frac{1}{2}} + \frac{2}{|\lambda_j|}\bigg),
\end{align*}
which is finite and uniformly bounded with respect to $T$. Thus, 
\begin{align*}
&\mathcal{F}A(\lambda_1,\ldots, \lambda_{m+n})= \lim_{T\to\infty}\mathcal{F}A_T(\lambda_1,\ldots, \lambda_{m+n})\\ 
& = 2\pi \mathcal{F}\phi(\lambda_1+ \ldots+ \lambda_m, \lambda_{m+1} + \ldots + \lambda_{m+n}) \prod_{j=1}^{m+n} \bigg(\frac{1}{\sqrt{2\pi}}\int_0^\infty e^{-iu\lambda_j}u^{H_0 -\frac{3}{2}}du \bigg).
\end{align*}
The integral inside the product is an improper Riemann integral. After the change of variables $v=u|\lambda_j|$, we get
\begin{align*}
\mathcal{F}A(&\lambda_1,\ldots, \lambda_{m+n})\\ 
& = 2\pi \mathcal{F}\phi(\lambda_1+ \ldots+ \lambda_m, \lambda_{m+1} + \ldots + \lambda_{m+n})\\
&\qquad\qquad\qquad\times \prod_{j=1}^{m+n} \bigg(|\lambda_j|^{\frac{1}{2}-H_0}\frac{1}{\sqrt{2\pi}}\int_0^\infty e^{-iu\text{sign} \lambda_j}u^{H_0 -\frac{3}{2}}du \bigg)\\
&= 2\pi \mathcal{F}\phi(\lambda_1+ \ldots+ \lambda_m, \lambda_{m+1} + \ldots + \lambda_{m+n}) \\
&\qquad\qquad\qquad\times\prod_{j=1}^{m+n} \bigg(|\lambda_j|^{\frac{1}{2}-H_0}\frac{1}{\sqrt{2\pi}}\Gamma(H_0 -\frac{1}{2})C(\lambda_j)\bigg),
\end{align*}
where $C(\lambda) = e^{-i\frac{\pi}{2}(H_0 - \frac{1}{2})}$ for $\lambda > 0, C(-\lambda) =\overline{C(\lambda)}$ and thus $|C(\lambda)|=1$ for all $\lambda \ne 0$, see appendix for the detailed computations. Applying Lemma \ref{Lemma6.1} by noticing that $C(\lambda_j)W(d\lambda_j) \overset{(d)}{=} W(d\lambda_j)$ (see \cite[Proposition 4.2]{Dobrushin}) and symmetrizing the Fourier transform of $A(\lambda_1, \ldots, \lambda_{m+n})$ lead to the desired conclusion.
\end{proof}

\subsection{Stochastic representations of the Rosenblatt process}
Let $(R^H(t))_{t \geq 0}$ be the Rosenblatt process of parameter $H \in (\frac{1}{2}, 1)$. The time representation of $R^H$ is 
\begin{align*}
R^H(t) &= a_1(D)\int_{\mathbb{R}^2}\bigg(\int_0^t (s-\xi_1)_+^{D-\frac{3}{2}}(s-\xi_2)_+^{D-\frac{3}{2}}ds\bigg)dB(\xi_1)dB(\xi_2)\\
 &=A_1(H)\int_{\mathbb{R}^2}\bigg(\int_0^t (s-\xi_1)_+^{\frac{H}{2}-1}(s-\xi_2)_+^{\frac{H}{2}-1}ds\bigg)dB(\xi_1)dB(\xi_2), 
\end{align*}
where $D = \frac{H+1}{2}$ and
$$a_1(D):= \frac{\sqrt{(D-1/2)(4D-3)}}{\beta(D-1/2, 2-2D)} = \frac{\sqrt{(H/2)(2H-1)}}{\beta(H/2, 1-H)}=:A_1(H).$$
Observe also that $1/2 < H<1 \Longleftrightarrow 3/4 < D < 1$. The corresponding spectral representation of this process, see for instance \cite{Taqqu, Taqqu2} or apply Lemma \ref{Lemma2}, is given by
\begin{align*}
R^H(t) &= a_2(D)\int_{\mathbb{R}^2}|\lambda_1|^{\frac{1}{2} - D}|\lambda_2|^{\frac{1}{2} - D}\frac{e^{i(\lambda_1+\lambda_2)t}-1}{i(\lambda_1+\lambda_2)}W(d\lambda_1)W(d\lambda_2)\\
 &=A_2(H)\int_{\mathbb{R}^2}|\lambda_1|^{-\frac{H}{2}}|\lambda_2|^{-\frac{H}{2}}\frac{e^{i(\lambda_1+\lambda_2)t}-1}{i(\lambda_1+\lambda_2)}W(d\lambda_1)W(d\lambda_2),
\end{align*}
where
$$a_2(D):= \sqrt{\frac{(2D-1)(4D-3)}{2[2\Gamma(2-2D)\sin (\pi (D-1/2))]^2}} = \sqrt{\frac{H(2H-1)}{2[2\Gamma(1-H)\sin (H\pi/2)]^2}}=:A_2(H).$$

\section{Proof of Theorem \ref{Theorem1}}
We are now in a position to give the proof of our Theorem \ref{Theorem1}. It is devided into four steps.

\subsection{Chaotic decomposition}

Using (\ref{eq:13}), we can write $X^{(q, H)}$ as a $q$-th Wiener-It\^{o} integral with respect to the standard two-sided Brownian motion $(B_t)_{t \in \mathbb{R}}$ as follows:
\begin{equation}\label{eq:5}
X^{(q, H)}_t = \int_{\mathbb{R}^q} L(x, t)(\xi_1,\ldots, \xi_q)dB(\xi_1)\ldots dB(\xi_q) = I_q^B(L(x, t)),
\end{equation}
where
\begin{equation}\label{eq:6}
L(x, t)(\xi_1,\ldots, \xi_q): = c(H, q) \int_{\mathbb{R}}\mathbf{1}_{[0, t]}(z) x(t -z) \prod_{j=1}^q (z - \xi_j)_+^{H_0 - \frac{3}{2}}dz,
\end{equation}
 with $c(H, q)$ and $H_0$ given by (\ref{eq:H2}). Applying the product formula (\ref{eq:P1}) for multiple Wiener-It\^{o} integrals, we easily obtain that
\begin{equation}\label{eq:dola2}
(X_t^{(q, H)})^2 - E[(X_t^{(q, H)})^2] = \sum_{r=0}^{q-1}r!\binom{q}{r}^2I_{2q-2r}^B(L(x, t) \widetilde{\otimes}_r L(x, t)).
\end{equation}
{\allowdisplaybreaks
Let us compute the contractions appearing in the right-hand side of (\ref{eq:dola2}). For every $ 0 \leq r \leq q-1$, by using Fubini's theorem we first have
\begin{align*}
L(x,& s) \otimes_r L(x, s) (\xi_1, \ldots, \xi_{2q-2r})\\ 
& = \int_{\mathbb{R}^r}dy_1 \ldots dy_r L(x, s)(\xi_1, \ldots, \xi_{q-r}, y_1, \ldots, y_r)L(x, s)(\xi_{q-r+1}, \ldots, \xi_{2q-2r}, y_1, \ldots, y_r)\\
&= c(H, q)^2 \int_{\mathbb{R}^r}dy_1 \ldots dy_r \int_0^s dz_1x(s-z_1) \prod_{j=1}^{q-r}(z_1 - \xi_j)_+^{H_0 - \frac{3}{2}}\prod_{i=1}^r(z_1 - y_i)_+^{H_0 - \frac{3}{2}}\\
& \qquad\qquad\quad\qquad\qquad\times  \int_0^s dz_2x(s-z_2) \prod_{j=q-r+1}^{2q-2r}(z_2 - \xi_j)_+^{H_0 - \frac{3}{2}}\prod_{i=1}^r(z_2 - y_i)_+^{H_0 - \frac{3}{2}}\\
& = c(H, q)^2 \int_{[0, s]^2}dz_1dz_2 x(s-z_1) x(s-z_2) \prod_{j=1}^{q-r}(z_1 - \xi_j)_+^{H_0 - \frac{3}{2}}\prod_{j=q-r+1}^{2q-2r}(z_2 - \xi_j)_+^{H_0 - \frac{3}{2}}\\
&\qquad\qquad\qquad\qquad\quad\times \bigg(\int_{\mathbb{R}}dy (z_1 - y)_+^{H_0 - \frac{3}{2}} (z_2 - y)_+^{H_0 - \frac{3}{2}}\bigg)^r,
\end{align*}
}and, since for any $z_1, z_2 \geq 0$
\begin{equation}\label{eq:2}
 \int_{\mathbb{R}}(z_1-y)_+^{H_0 - \frac{3}{2}}(z_2-y)_+^{H_0 -\frac{3}{2}}dy = \beta\Big(H_0 - \frac{1}{2}, 2-2H_0\Big)|z_1-z_2|^{2H_0-2},
\end{equation}
we end up with the following expression
 \begin{align}\label{1}
&L(x, s) \otimes_r L(x, s) (\xi_1, \ldots, \xi_{2q-2r}) \nonumber \\ 
& = c(H, q)^2\beta\Big(H_0 - \frac{1}{2}, 2-2H_0\Big)^r\int_{[0, s]^2}dz_1dz_2 x(s-z_1) x(s-z_2)|z_1 -z_2|^{(2H_0-2)r} \nonumber \\
&\qquad\qquad\qquad\qquad\qquad\qquad\qquad\times \prod_{j=1}^{q-r}(z_1 - \xi_j)_+^{H_0 - \frac{3}{2}}\prod_{j=q-r+1}^{2q-2r}(z_2 - \xi_j)_+^{H_0 - \frac{3}{2}}.
 \end{align}
Recall $G_T^{(q, H)}$ from (\ref{eq:Gt}). As a consequence, we can write 
\begin{equation}\label{eq:8}
G_T^{(q, H)} (t)= F_{2q, T}(t) + c_{2q-2}F_{2q-2, T}(t) + \ldots + c_4F_{4, T}(t) + c_2F_{2, T}(t)
\end{equation}
where $c_{2q-2r}:= r!\binom{q}{r}^2$ and for $ 0 \leq r \leq q-1$,
\begin{equation}\label{eq:7}
F_{2q-2r, T}(t): = \frac{1}{T^{2H_0 -1}}\int_0^{Tt} I_{2q-2r}^B(L(x,s) \widetilde{\otimes}_r L(x, s))ds,
\end{equation}
where the kernels in each Wiener integral above are given explicitly in (\ref{1}).

\subsection{Spectral representations}

Recall the expression of the contractions $L(x, s) \otimes_r L(x, s), 0\leq r \leq q-1$ given in (\ref{1}). Set
\begin{align*}
\phi_r(s, z_1, z_2) :=&c(H, q)^2\beta\Big(H_0 - \frac{1}{2}, 2-2H_0\Big)^r\\
&\times  \mathbf{1}_{[0, s]}(z_1)\mathbf{1}_{[0, s]}(z_2)x(s-z_1)x(s-z_2) |z_1-z_2|^{(2H_0-2)r}.
\end{align*}
It is a symmetric function with respect to $z_1$ and $z_2$. Furthermore, by H\"{o}lder's inequality, we have
\begin{align*}
&\int_{\mathbb{R}^2}\Big|\mathbf{1}_{[0, s]}(z_1)\mathbf{1}_{[0, s]}(z_2)x(s-z_1)x(s-z_2) |z_1-z_2|^{(2H_0-2)r}\Big|dz_1dz_2\\
&\leq \int_{[0, s]^2} |x(s-z_1)| |x(s-z_2)| |z_1-z_2|^{(2H_0-2)r}dz_1dz_2 \\
&= \int_{[0, s]^2} |x(z_1)| |x(z_2)| |z_1-z_2|^{r\frac{(2H-2)}{q}}dz_1dz_2\\
&\leq \bigg(\int_{[0, \infty)^2}|x(z_1)||x(z_2)||z_1 -z_2|^{2H-2}dz_1dz_2 \bigg)^{\frac{r}{q}}\bigg(\int_0^\infty |x(z)|dz\bigg)^{2(1-\frac{r}{q})}.
\end{align*}
Using the integrability of $x$ together with the assumption (\ref{eq:1}), it turns out that $\phi_r(. , z_1, z_2) $ is integrable on $\mathbb{R}^2_+$. Applying Lemma \ref{Lemma2} with $m=n=q-r$, we get
\begin{align*}
F_{2q-2r, T}(t) &= \frac{1}{T^{2H_0 - 1}}\int_0^{Tt} I_{2q-2r}^B(L(x, s) \widetilde{\otimes}_r L(x, s))ds \nonumber\\
&  \overset{(d)}{=} A_r(H, q) \frac{1}{T^{2H_0 -1}} \int_{\mathbb{R}^{2q-2r}}W(d\lambda_1)\ldots W(d\lambda_{2q-2r}) \prod_{j=1}^{2q-2r}|\lambda_j|^{\frac{1}{2}-H_0} \nonumber\\
&\times \frac{1}{(2q-2r)!}\sum_{\sigma \in \mathfrak{S}_{2q-2r}}\int_0^{Tt}ds \int_{[0, s]^2}d\xi_1d\xi_2 x(s-\xi_1) x(s-\xi_2)|\xi_1 -\xi_2|^{(2H_0-2)r} \nonumber\\
&\qquad\qquad\qquad\qquad\qquad\qquad\times e^{i(\lambda_{\sigma(1)} + \ldots+ \lambda_{\sigma(q-r)})\xi_1}e^{i(\lambda_{\sigma(q-r+1)} + \ldots + \lambda_{\sigma(2q-2r)})\xi_2},
 \end{align*}
where
\begin{equation}\label{eq:Ar}
A_r(H,q) := c(H, q)^2 \beta(H_0 - \frac{1}{2}, 2- 2H_0)^r \bigg( \frac{\Gamma(H_0 - \frac{1}{2})}{\sqrt{2\pi}}\bigg)^{2q-2r}.
\end{equation}
The change of variable $s = Ts'$ yields
\begin{align*}
F_{2q-2r, T}(t) & \overset{(d)}{=} A_r(H, q)T^{2-2H_0} \int_{\mathbb{R}^{2q-2r}}W(d\lambda_1)\ldots W(d\lambda_{2q-2r}) \prod_{j=1}^{2q-2r}|\lambda_j|^{\frac{1}{2}-H_0}\nonumber\\
&\times \frac{1}{(2q-2r)!}\sum_{\sigma \in \mathfrak{S}_{2q-2r}}\int_0^{t}ds \int_{[0, Ts]^2}d\xi_1d\xi_2 x(Ts-\xi_1) x(Ts-\xi_2)|\xi_1 -\xi_2|^{(2H_0-2)r} \nonumber\\
&\qquad\qquad\qquad\qquad\qquad\qquad\times e^{i(\lambda_{\sigma(1)} + \ldots+ \lambda_{\sigma(q-r)})\xi_1}e^{i(\lambda_{\sigma(q-r+1)} + \ldots + \lambda_{\sigma(2q-2r)})\xi_2}.
 \end{align*}
Let us do a further change of variables: $\lambda'_{\sigma(j)} = T\lambda_{\sigma(j)}, j = 1, \ldots, 2q-2r $ and $\xi'_k = Ts -  \xi_k, k=1,2$. Thanks to the self-similarity of $W$ with index $1/2$ (that is, $W(T^{-1}d\lambda)$ has the same law as $T^{-1/2}W(d\lambda)$) we finally obtain that
\begin{align}\label{4}
F_{2q-2r, T}(t) & \overset{(d)}{=} A_r(H, q)T^{-(2-2H_0)(q-1-r)} \nonumber \\
& \times\int_{\mathbb{R}^{2q-2r}}W(d\lambda_1)\ldots W(d\lambda_{2q-2r}) \prod_{j=1}^{2q-2r}|\lambda_j|^{\frac{1}{2}-H_0}\int_0^{t}ds e^{i(\lambda_1 + \ldots + \lambda_{2q-2r})s}\nonumber\\
&\times \frac{1}{(2q-2r)!}\sum_{\sigma \in \mathfrak{S}_{2q-2r}} \int_{[0, Ts]^2}d\xi_1d\xi_2 x(\xi_1) x(\xi_2)|\xi_1 -\xi_2|^{(2H_0-2)r}  \nonumber\\
&\qquad\qquad\qquad\qquad\quad\times e^{-i (\lambda_{\sigma(1)} + \ldots+ \lambda_{\sigma(q-r)})\frac{\xi_1}{T}}e^{-i (\lambda_{\sigma(q-r+1)} + \ldots + \lambda_{\sigma(2q-2r)})\frac{\xi_2}{T}}.
 \end{align}

\subsection{Reduction lemma}

\begin{lemma}\label{Reduction}
Fix $t$, fix $H \in (\frac{1}{2}, 1)$ and fix $q \geq 2$. Assume (\ref{eq:1}) and  the integrability of the kernel $x$. Then for any $r \in \{0, \ldots, q-2 \}$, one has
$$\lim_{T \to \infty} E[F_{2q-2r, T}(t)^2]= 0.$$ 
\end{lemma}

\begin{proof}
Without loss of generality, we may and will assume that $t=1$. From the spectral representation of multiple Wiener-It\^{o} integrals (\ref{4}), one has 
\eject
\begin{align*}
&E[F_{2q-2r, T}(1)^2]   \\
&=  T^{-2(2-2H_0)(q-1-r)}  A_r^2(H, q)(2q-2r)!\int_{\mathbb{R}^{2q-2r}} d\lambda_1 \ldots d\lambda_{2q-2r} \prod_{j=1}^{2q-2r}|\lambda_j|^{1-2H_0}\\
&\times \bigg( \frac{1}{(2q-2r)!}\sum_{\sigma \in \mathfrak{S}_{2q-2r}}\int_0^{1}ds e^{i(\lambda_1 + \ldots + \lambda_{2q-2r})s} \int_{[0, Ts]^2}d\xi_1d\xi_2 x(\xi_1) x(\xi_2)|\xi_1 -\xi_2|^{(2H_0-2)r} \\
 &\hspace{6cm}\times e^{-i (\lambda_{\sigma(1)} + \ldots+ \lambda_{\sigma(q-r)})\frac{\xi_1}{T}}e^{-i (\lambda_{\sigma(q-r+1)} + \ldots + \lambda_{\sigma(2q-2r)})\frac{\xi_2}{T}} \bigg)^2.
\end{align*}
Since $x$ is a real-valued integrable function on $[0, \infty)$ satisfying assumption (\ref{eq:1}), we deduce from Lebesgue dominated convergence that, as $T \to \infty$,
\begin{align*}
& \frac{1}{(2q-2r)!}\sum_{\sigma \in \mathfrak{S}_{2q-2r}}\int_0^{1}ds e^{i(\lambda_1 + \ldots + \lambda_{2q-2r})s} \int_{[0, Ts]^2}d\xi_1d\xi_2 x(\xi_1) x(\xi_2)|\xi_1 -\xi_2|^{(2H_0-2)r} \\
 &\hspace{6cm} \times e^{-i (\lambda_{\sigma(1)} + \ldots+ \lambda_{\sigma(q-r)})\frac{\xi_1}{T}}e^{-i (\lambda_{\sigma(q-r+1)} + \ldots + \lambda_{\sigma(2q-2r)})\frac{\xi_2}{T}}\\
&\longrightarrow \int_{[0, \infty)^2}x(u)x(v)|u-v|^{(2H_0-2)r}dudv \int_0^1 e^{i(\lambda_1+\ldots + \lambda_{2q-2r})s}ds.
\end{align*}
Since $1- \frac{1}{2q} < H_0 < 1$ and $0 \leq r \leq q-2$, we have $ T^{-2(2-2H_0)(q-1-r)} \to 0 $ as $T \to \infty$. Moreover, since $\int_0^1 e^{i(\lambda_1+\ldots + \lambda_{2q-2r})\xi}d\xi = \frac{e^{i(\lambda_1+\ldots + \lambda_{2q-2r})} -1}{i(\lambda_1 + \ldots + \lambda_{2q-2r})}$,
\begin{align*}
&\int_{\mathbb{R}^{2q-2r}} d\lambda_1 \ldots d\lambda_{2q-2r} \prod_{j=1}^{2q-2r}|\lambda_j|^{1-2H_0} \bigg|\frac{e^{i(\lambda_1+\ldots + \lambda_{2q-2r})} -1}{i(\lambda_1 + \ldots + \lambda_{2q-2r})}\bigg|^2 \leq \bigg(\int_{\mathbb{R}} |\lambda|^{1-2H_0}d\lambda\bigg)^{2q-2r}
\end{align*}
which is integrable at zero, and
\begin{align*}
&\int_{\mathbb{R}^{2q-2r}} d\lambda_1 \ldots d\lambda_{2q-2r} \prod_{j=1}^{2q-2r}|\lambda_j|^{1-2H_0} \bigg|\frac{e^{i(\lambda_1+\ldots + \lambda_{2q-2r})} -1}{i(\lambda_1 + \ldots + \lambda_{2q-2r})}\bigg|^2\\ 
& \leq \int_{\mathbb{R}^{2q-2r}} d\lambda_1 \ldots d\lambda_{2q-2r} \prod_{j=1}^{2q-2r}|\lambda_j|^{1-2H_0} \frac{4}{(\lambda_1 + \ldots + \lambda_{2q-2r})^2}
\end{align*}
which is integrable at infinity, we have
$$\int_{\mathbb{R}^{2q-2r}} d\lambda_1 \ldots d\lambda_{2q-2r} \prod_{j=1}^{2q-2r}|\lambda_j|^{1-2H_0} \bigg|\frac{e^{i(\lambda_1+\ldots + \lambda_{2q-2r})} -1}{i(\lambda_1 + \ldots + \lambda_{2q-2r})}\bigg|^2 < \infty.$$
All these facts taken together imply 
\begin{equation}\label{eq:10}
E[F_{2q-2r, T}(1)^2] \longrightarrow 0, \text{ as } T \to \infty, \text{ for all } 0 \leq r \leq q-2,
\end{equation}
which proves the lemma.
\end{proof}

\subsection{Concluding the proof of Theorem \ref{Theorem1}}
Thanks to Lemma \ref{Reduction}, we are left to concentrate on the convergence of the term $F_{2, T}$ (belonging to the second Wiener chaos) corresponding to $r=q-1$. Recall from (\ref{4}) that $F_{2, T}(t)$ has the same law as the double Wiener integral with symmetric kernel given by
\begin{align}\label{eq:15}
f_T(t, \lambda_1, &\lambda_2):= A_{q-1}(H,q) |\lambda_1|^{\frac{1}{2} - H_0}|\lambda_2|^{\frac{1}{2} - H_0}\int_0^t ds e^{i(\lambda_1 +\lambda_2)s} \nonumber\\ 
&\qquad\quad  \times  \int_{[0, Ts]^2}d\xi_1d\xi_2 e^{-i(\lambda_1\frac{\xi_1}{T} + \lambda_2\frac{\xi_2}{T})}x(\xi_1)x(\xi_2)|\xi_1 - \xi_2|^{(q-1)(2H_0 - 2)}.
\end{align}
Observe that $f_T(t, .)$ is symmetric, so there is no need to care about symmetrization. By the isometry property of multiple Wiener-It\^{o} integrals with respect to the random spectral measure, in order to prove the $L^2(\Omega)$-convergence of $c_2F_{2, T}$ to $bR^{H'}$, we can equivalently prove that $c_2f_T(t, .)$ converges in $L^2(\mathbb{R}^2)$ to the kernel of $bR^{H'}(t)$ . First, by Lebesgue dominated convergence, as $T \to \infty$, we have
\begin{align*}
f_T(t, \lambda_1, \lambda_2)& \longrightarrow A_{q-1}(H, q) \int_{\mathbb{R}^2}x(u)x(v)|u-v|^{(q-1)(2H_0 - 2)}dudv\\ 
&\qquad\qquad\qquad\qquad\qquad\times  |\lambda_1|^{\frac{1}{2} - H_0}|\lambda_2|^{\frac{1}{2} - H_0} \frac{e^{i(\lambda_1 + \lambda_2)t} -1}{i(\lambda_1 + \lambda_2)}.
\end{align*}
This shows that $f_T(t, .)$ converges pointwise to the kernel of $R^{H'}(t)$, up to some constant. 
Moreover, for all $0 < S <T $,
\begin{align*}
&\|f_T(t, .) - f_S(t, .)\|_{L^2(\mathbb{R}^2)}^2\\
& = A_{q-1}^2(H, q) \int_{\mathbb{R}^2}d\lambda_1 d\lambda_2 |\lambda_1|^{1 - 2H_0}|\lambda_2|^{1 - 2H_0}\\ 
&\quad \times \bigg( \int_0^t ds e^{i(\lambda_1 +\lambda_2)s}\int_{[0, Ts]^2 \setminus  [0, Ss]^2}d\xi_1d\xi_2 e^{-i(\lambda_1\frac{\xi_1}{T} + \lambda_2\frac{\xi_2}{T})}x(\xi_1)x(\xi_2)|\xi_1 - \xi_2|^{(q-1)(2H_0 - 2)}\bigg)^2.
\end{align*}
By Lebesgue dominated convergence, it comes that $\|f_T(t, .) - f_S(t, .)\|_{L^2(\mathbb{R}^2)}^2 \longrightarrow 0$ as $T, S \to \infty$. It follows that $(f_T(t, .))_{T \geq 0}$ is a Cauchy sequence in $L^2(\mathbb{R}^2)$. Hence, the multiple Wiener integral $c_2F_{2, T}$ (with kernel (\ref{eq:15})) converges in $L^2(\Omega)$ to $b(H, q) \times R^{H'}$ with the explicit constant $b(H, q)$ as in (\ref{eq:19}). (Note that $c_2 =q!$). The finite-dimensional convergence then follows from (\ref{4}). The proof of Theorem \ref{Theorem1} is achieved.
\qed

\section{Proof of Theorem \ref{Theorem2}}

We follow the same route as for the proof of Theorem \ref{Theorem1}, with some slight modifications. Here, the chaos decomposition of $G_T^{(1, H)}$ contains uniquely the term $F_{2, T}$ obtained for $q=1$ and $ r = 0$. Its spectral representation is as follows:
\begin{align*}
F_{2, T}(t)&= \frac{H(2H-1)}{\beta(H-\frac{1}{2}, 2-2H)}\frac{\Gamma^2(H-\frac{1}{2})}{2\pi} \int_{\mathbb{R}^2}W(d\lambda_1)W(d\lambda_2) |\lambda_1|^{\frac{1}{2} - H}|\lambda_2|^{\frac{1}{2} - H}\nonumber\\ 
&\qquad\qquad\qquad  \times \int_0^t ds e^{i(\lambda_1 +\lambda_2)s}  \int_{[0, Ts]^2}d\xi_1d\xi_2 e^{-i(\lambda_1\frac{\xi_1}{T} + \lambda_2\frac{\xi_2}{T})}x(\xi_1)x(\xi_2).
\end{align*}
It is easily seen that that $F_{2, T}$ is well-defined if and only if $3/4 < H <1$. The same arguments as in the proof of Theorem \ref{Theorem1} yield
\begin{equation}\label{eq:d}
G_T^{(1, H)}(t) = F_{2, T}(t) \longrightarrow \frac{H(2H-1)}{\sqrt{(H-1/2)(4H-3)}}\bigg(\int_0^\infty x(u)du\bigg)^2 \times R^{H''}(t)
\end{equation}
in $L^2(\Omega)$ as $T \to \infty$, thus completing the proof of the theorem. \qed

\section*{Acknowledgements} 
I would like to sincerely thank my supervisor Ivan Nourdin, who has led the way on this work. I did appreciate his advised tips and encouragement for my first research work. I also warmly thank Frederi Viens for interesting discussions and several helpful comments about this work. Also, I would like to  thank my friend, Nguyen Van Hoang, for his help in proving the identity about $I$ in the appendix. Finally, I deeply thank an anonymous referee for a very careful and thorough reading of this work, and for her/his constructive remarks.

\section*{Appendix}
 
The following identity has been used at the end of the proof of Lemma \ref{Lemma2} and also appeared in the proof of \cite[Lemma 6.2]{Taqqu}. 

For all $H_0 \in (1/2, 1)$, we have
$$I: = \int_0^\infty e^{-iu}u^{H_0 - \frac{3}{2}}du =  e^{-i\frac{\pi}{2}(H_0 - \frac{1}{2})} \Gamma(H_0 - \frac{1}{2}).$$
\begin{proof}
First, observe that
$$u^{H_0 - \frac{3}{2}} = \frac{1}{\Gamma(\frac{3}{2} - H_0)} \int_0^\infty e^{-tu}t^{\frac{1}{2} - H_0}dt.$$
Then, Fubini's theorem yields
\begin{align*}
I& =\frac{1}{\Gamma(\frac{3}{2} - H_0)} \int_0^\infty du e^{-iu}  \int_0^\infty dt e^{-tu}t^{\frac{1}{2} - H_0}\\ 
& = \frac{1}{\Gamma(\frac{3}{2} - H_0)} \int_0^\infty dt \text{ } t^{\frac{1}{2} - H_0} \int_0^\infty du e^{-u(t+i)}\\
&=  \frac{1}{\Gamma(\frac{3}{2} - H_0)} \int_0^\infty  t^{\frac{1}{2} - H_0} \frac{1}{t+i} dt=  \frac{1}{\Gamma(\frac{3}{2} - H_0)} \int_0^\infty  \frac{t^{\frac{1}{2} - H_0}(t-i)}{t^2+1}dt\\
& = \frac{1}{\Gamma(\frac{3}{2} - H_0)} \bigg( \int_0^\infty  \frac{t^{\frac{3}{2} - H_0}}{t^2+1}dt - i \int_0^\infty  \frac{t^{\frac{1}{2} - H_0}}{t^2+1}dt \bigg).
\end{align*}
A change of variables $t = \sqrt{u}$ and $v= \frac{u}{u+1}$ leads to 
\begin{align*}
\int_0^\infty  \frac{t^{\frac{3}{2} - H_0}}{t^2+1}dt& = \frac{1}{2}\int_0^\infty \frac{u^{\frac{1-2H_0}{4}}}{u+1}du =  \frac{1}{2} \int_0^1 v^{\frac{1-2H_0}{4}}(1-v)^{\frac{2H_0 -  5}{4}} \\ 
& =  \frac{1}{2}  \beta \Big(\frac{5-2H_0 }{4}, \frac{2H_0 -  1}{4}\Big) = \frac{1}{2} \frac{\Gamma(\frac{5-2H_0 }{4})\Gamma(\frac{2H_0 -1}{4})}{\Gamma(1)}.
\end{align*}
Similarly, one also has,
$$\int_0^\infty  \frac{t^{\frac{1}{2} - H_0}}{t^2+1}dt = \frac{1}{2}  \beta \Big(\frac{3-2H_0 }{4}, \frac{2H_0 +  1}{4}\Big) = \frac{1}{2} \frac{\Gamma(\frac{3-2H_0 }{4})\Gamma(\frac{2H_0 +1}{4})}{\Gamma(1)}.$$
Furthermore, by using the identity $\Gamma(1-z)\Gamma(z) = \frac{\pi}{\sin(\pi z)}, 0<z<1$, we obtain
\begin{align*}
I&=\frac{1}{2 \Gamma(\frac{3}{2} - H_0)}   \bigg(\frac{\pi}{\sin (\frac{2H_0 -1}{4} \pi)} - i \frac{\pi}{\sin (\frac{3-2H_0}{4} \pi)}  \bigg)\\ 
& = \frac{1}{2 \Gamma(\frac{3}{2} - H_0)}   \bigg(\frac{\pi}{\sin (\frac{2H_0 -1}{4} \pi)} - i \frac{\pi}{\cos (\frac{2H_0-1}{4} \pi)}  \bigg)\\
&= \frac{\pi}{ \Gamma(\frac{3}{2} - H_0)} \frac{e^{-i\frac{\pi}{2}(H_0 - \frac{1}{2})}}{2\sin (\frac{2H_0 -1}{4} \pi)\cos(\frac{2H_0 -1}{4} \pi)}\\
& = \frac{e^{-i\frac{\pi}{2}(H_0 - \frac{1}{2})} \pi}{ \Gamma(\frac{3}{2} - H_0) \sin (\frac{2H_0 -1}{2} \pi)} = e^{-i\frac{\pi}{2}(H_0 - \frac{1}{2})} \Gamma(H_0 - \frac{1}{2}).
\end{align*}

\end{proof}

\end{document}